\documentclass[10pt,a4paper,reqno]{amsart}

\usepackage{amssymb}
\usepackage{enumitem}
\usepackage[ruled,noline,titlenumbered,linesnumbered,noend,norelsize]{algorithm2e}
\DontPrintSemicolon
\SetNlSty{footnotesize}{}{}
\IncMargin{0.5em}
\SetNlSkip{1em}
\SetKwIF{If}{ElseIf}{Else}{if}{then}{else if}{else}{endif}
\SetKw{Return}{return}
\SetKw{Pick}{pick}
\SetNlSkip{2em}
\newenvironment{algorithm-hbox}{\hbadness=10000\begin{algorithm}}{\end{algorithm}}

\usepackage[%draft, % do not do any hyper linking,
    pdftex, colorlinks=true, linkcolor=black, urlcolor=black, citecolor=black,
    pdfpagelabels, bookmarks, bookmarksnumbered, naturalnames=true,
	pdfkeywords={hypergraph, greedy coloring, property B},
    bookmarksopen, breaklinks, linktocpage,
		pdftitle={A note on random greedy coloring of uniform hypergraphs}, pdfdisplaydoctitle,
]{hyperref}

\theoremstyle{plain}
\newtheorem{theorem}{Theorem}

\newtheorem{corollary}[theorem]{Corollary}

\numberwithin{equation}{section}

\newenvironment{enumeratei-continued}{\begin{enumerate}[label=\textup{(\roman*)}, noitemsep, topsep=1.5mm, labelindent=.8em, leftmargin=*, widest=., resume=my]}{\end{enumerate}}

\renewcommand{\epsilon}{\varepsilon}

\renewcommand{\leq}{\leqslant}
\renewcommand{\geq}{\geqslant}

\newcommand{\Nat}{\mathbb{N}}

\newcommand{\evB}{\textbf{B}}
\newcommand{\evP}{\textbf{P}}
\newcommand{\evR}{\textbf{R}}

\begin{document}
\title[A note on random greedy coloring of uniform hypergraphs]{A note on random greedy coloring of uniform hypergraphs}

\author{Danila D. Cherkashin}
\address{Saint-Petersburg State University, Faculty of Mathematics and Mechanics, Saint-Petersburg, Russia}
\email{matelk@mail.ru}

\author{Jakub Kozik}
\address{Theoretical Computer Science Department, Faculty of Mathematics and
Computer Science, Jagiellonian University, Krak\'{o}w, Poland}
\thanks{Research of J.\ Kozik  was supported by Polish National
Science Center within grant 2011/01/D/ST1/04412.}
\email{Jakub.Kozik@uj.edu.pl}

\begin{abstract}
The smallest number of edges forming an $n$-uniform hypergraph
which is not $r$-colorable is denoted by $m(n,r)$. Erd\H{o}s and Lov\'{a}sz conjectured that 
$m(n,2)= \theta\left(n 2^n\right)$.
The best known lower bound $m(n,2) = \Omega\left(\sqrt{n/\log n} 2^n\right)$ was obtained by Radhakrishnan and Srinivasan 
in 2000. We present a simple proof of their result. 
The proof is based on analysis of random greedy coloring algorithm investigated by Pluh\'ar in 2009. 
The proof method extends to the case of $r$-coloring, and we show that for any fixed $r$ we have 
$m(n,r)=  \Omega\left(\left(n/\log n\right)^{(r-1)/r} \; r^n\right)$ 
improving the bound of Kostochka from 2004. We also derive analogous bounds on minimum edge degree of an 
$n$-uniform hypergraph that is not $r$-colorable.

\end{abstract}

\maketitle

\newcommand{\Red}{red}
\newcommand{\Blue}{blue}

\newcommand{\prob}{\mathbb{P}}
\newcommand{\expt}{\mathbb{E}}

A hypergraph is a pair $(V,E)$, where $V$ is a set of vertices and $E$ is a family of subsets of $V$. 
Hypergraph is \emph{$n$-uniform} if all its edges have exactly $n$ elements. Hypergraph $(V,E)$ is 
\emph{$r$ colorable} if there exists a coloring of vertices with $r$ colors in which no edge is monochromatic 
(i.e., there exists a function $c:V \to \{1,\ldots,r\}$ such that the image of every edge has at least two elements). 
Hypergraph has \emph{property B} if it is two-colorable. For $n,r\in \Nat$ let $m(n,r)$ be the smallest number of 
edges of an $n$-uniform hypergraph that is not $r$-colorable. The asymptotic behaviour of $m(n)=m(n,2)$ was first 
studied by Erd{\H{o}}s. In \cite{Er1963} and \cite{Er1964} Erd{\H{o}}s proved that:
\[
	2^{n-1} \leq m(n) \leq (1+o(1)) \frac{e \ln{2}}{4} n^2 2^n.
\]
In \cite{EL1975} Erd{\H{o}}s and Lov\'asz wrote that ``perhaps $n2^n$ is the correct order of magnitude of $m(n)$''. 
The upper bound has not been improved since. The most recent improvement on the lower bound was obtained by 
Radhakrishnan and Srinivasan in \cite{RS00}.
We present a simple proof of their main theorem:
\begin{theorem}[Radhakrishnan, Srinivasan 2000]
\label{thm:RS}
\[
	m(n) = \Omega\left( \left(\frac{n}{\ln(n)}\right)^{1/2}  2^n\right).
\]
\end{theorem}
In fact we prove that, for $c< \sqrt{2}$ and  all sufficiently large $n$, whenever an $n$-uniform hypergraph has 
at most $ c \sqrt{n/\ln(n)} 2^{n-1}$ edges, then a simple random greedy algorithm produces a proper coloring 
with positive probability. The same coloring procedure was considered by Pluh\'ar in \cite{Pluhar09}, 
where a bound $m(n)=\Omega\left(n^{1/4}2^n\right)$ was obtained in an elegant and straightforward way.

The proof technique extends easily to the more general case of $r$-coloring 
(very much along the lines of development of Pluh\'ar \cite{Pluhar09}). 
To avoid technicalities we focus on asymptotics of $m(n,r)$ for fixed $r$ and $n$ tending to infinity.
\begin{theorem}
\label{thm:rCol}
For any fixed integer $r\geq 2$, we have
\[
	m(n,r)= \Omega\left( \left(\frac{n}{\ln(n)}\right)^{\frac{r-1}{r}} r^n\right)
\]
\end{theorem}
This improves the bounds of Kostochka \cite{Kos2004} which are of the order 
$(\frac{n}{\ln(n)})^{\frac{\lfloor \log_2(r) \rfloor}{\lfloor \log_2(r) \rfloor+1}} r^n$ and the bound $m(n,3)= \Omega(n^{1/2} 3^{n-1})$ by Shabanov \cite{Sha2012}.
% Some other recent bounds on  $m(n,r)$ which are also surpassed by our result can be found in \cite{Sh1}, \cite{Sh2}, \cite{Rai1}. 
Several other variants of  extremal problems on hypergraph coloring can be found in a survey by Raigorodskii and Shabanov \cite{Rai1}.

Just like the results from \cite{RS00} our results extend to a local version. 
Let $D(n,r)$ be the maximum number such that  every $n$-uniform hypergraph 
with strictly smaller edge degrees is $r$-colorable.

\begin{theorem}  For any fixed $r\geq 2$ we have
\label{thm:loc}
\[
	D(n,r) = \Omega\left( \left(\frac{n}{\ln(n)}\right)^{\frac{r-1}{r}} r^n\right).
\]
\end{theorem}

All the results of the paper are derived from analysis of random greedy $r$-coloring procedure listed as 
Algorithm \ref{alg}. 

\begin{algorithm-hbox}[!ht]
\caption{Random greedy $r$-coloring}\label{alg}
\ForEach{$v\in V$}{
 	\textup{choose uniformly at random a point $t(v)$ from interval $[0,1]$}
}
\textbf{let} $(v_1, \ldots, v_m)$ \textbf{be} $V$ ordered according to $t(v)$
(i.e. $t(v_i) \leq t(v_{i+1})$)  \;
\For{$i=1\ldots m$}{
	\uIf{$\exists_{j\in\{1,\ldots,r\}}$ such that coloring $v_i$ with $j$ does not create a 
	monochromatic edge with the last vertex $v_i$}{
		$c(v_i) \gets$ smallest such $j$
	}
	\uElse{
			$c(v_i) \gets r$
	}
}
\Return c
\end{algorithm-hbox}

Random value $t(v)$ assigned by the algorithm to a vertex $v$ will be called a 
\emph{birth time} of $v$. We assume that the birth time assignment function sampled by the algorithm is 
injective (this happens with probability 1). For any edge $f$, the \emph{first} (resp. \emph{last}) 
vertex of $f$ is the vertex $v\in f$ with smallest (largest) birth time.

\section{Property B}
\begin{proof}[Proof of Theorem \ref{thm:RS}]
Let $(V,E)$ be an $n$-uniform hypergraph with $k 2^{n-1}$ edges. Let us consider  random greedy 2-coloring algorithm 
\ref{alg}. Following a long tradition, we call colors $1,2$ respectively \emph{blue} and \emph{red}. 
Then the rule of assigning colors used by algorithm reduces to \emph{choose color blue unless the currently colored 
vertex is the last vertex of a blue edge}.
Every pair of edges $(e,f)$ such that the last vertex of $e$ is the first vertex of $f$ will be called a 
\emph{conflicting pair}.

Clearly there can not be monochromatic \Blue~ edges in the coloring constructed by the algorithm. 
Suppose that some edge $f$ is colored \Red, and let  $v$ be the first vertex of $f$. Vertex $v$ has been 
colored \Red~ by the algorithm, so there exists an edge $e$, such that $v$ is the last vertex of $e$. 
Edges $(e,f)$ form a conflicting pair.
By the above discussion if there are no conflicting pairs under the assignment $t$, 
then the coloring produced by the algorithm is proper. We are going to check for which values of $k$ 
the probability of having no conflicting pairs is positive.

Let us divide real interval $[0,1]$ into three subintervals  $B=\left[0,\frac{1-p}{2}\right), P=\left[ \frac{1-p}{2}, 
\frac{1+p}{2}\right),  R=\left[\frac{1+p}{2},1\right]$
(with  parameter $p$ to be optimized later). We consider three events:
\begin{description}
	\item[\evB] there exists a conflicting pair with  common vertex in $B$,
	\item[\evP] there exists a conflicting pair with  common vertex in $P$,
	\item[\evR] there exists a conflicting pair with  common vertex in $R$.
\end{description}
Clearly $\Pr[\evB]=\Pr[\evR]$ and they are both smaller than the probability that there exists an edge such 
that all its vertices have birth times from interval $B$. The expected number of such edges is 
$k 2^{n-1} (\frac{1-p}{2})^n$, hence:
\begin{equation}
\label{eq:BR}
	\Pr[\evB \cup \evR] \leq \Pr[\evB]+\Pr[\evR] = 2 \Pr[\evB] \leq k 2^n \left(\frac{1-p}{2}\right)^n = k (1-p)^n.
\end{equation}

A pair of edges $(e,f)$ with exactly one common vertex is called \emph{dangerous}. Only dangerous pairs can be conflicting. 
The probability that there exists a conflicting pair with common vertex in $P$ is bounded from above by the 
expected number of such pairs:
\begin{align*}
	\Pr[\evP] &\leq \expt[\sharp \text{ conflicting pairs with common vertex in P}]
\\
	&\leq (k 2^{n-1})^2 \Pr[\text{dangerous pair is conflicting with common vertex in P}]
\\
	 &\leq
	(k 2^{n-1})^2\int_{\frac{1-p}{2}}^{\frac{1+p}{2}} x^{n-1} (1-x)^{n-1}\mathrm{d}x
	\\
	 & =
	k^2 \int_{-\frac{p}{2}}^{\frac{p}{2}}  \left((1+2x)(1-2x)\right)^{n-1} \mathrm{d}x
\end{align*}
The integrand function is smaller than 1, so the length of the integration interval is an upper bound for the value 
of the integral. We get
\begin{equation}
\label{eq:P}
	\Pr[\evP] \leq k^2 p.
\end{equation}
Inequalities (\ref{eq:BR}) and (\ref{eq:P}) give:
\[
	\Pr[\evB \cup \evR \cup \evP] \leq \Pr[\evB \cup \evR] + \Pr[\evP] \leq k (1-p)^n + k^2 p.
\]
Hence, whenever the following inequality holds
\begin{equation}
\label{ineq:RS}
	k (1-p)^n + k^2 p < 1,
\end{equation}
the algorithm produces a proper coloring with positive probability. 

Let $k_n=c \sqrt{n/\ln(n)}$ and $p_n=\ln(n/k_n)/n$. Then
\[
	\lim_{n\to \infty}\left(k_n (1-p_n)^n + k_n^2 p_n\right) = c^2/2.
\]
Therefore for any $c< \sqrt{2}$ and all sufficiently large $n$, any $n$-uniform hypergraph with at most 
$c \sqrt{n/\ln(n)} 2^{n-1}$ edges has property B.
\end{proof}
\section{$r$- coloring}
\label{sec:rcol}
\begin{proof}[Proof of Theorem \ref{thm:rCol}]
Let $(V,E)$ be an $n$-uniform graph with $k \;r^{n-2}$ edges.
We analyse the probability that random greedy $r$-coloring procedure produces a proper coloring. 
Analogously to the work of Pluh\'ar \cite{Pluhar09} we focus on avoiding specific structures called conflicting $r$-chains.

A sequence of edges $(f_1, \ldots , f_r)$ is called an \emph{$r$-chain} if $|f_i \cap f_{i+1}|=1$ for $i=1, \ldots, r-1$, 
and $f_i \cap f_j =\emptyset$ for all $i,j\in \{1,\ldots r\}$ such that $|i-j|>1$. An $r$-chain is \emph{conflicting} 
under birth time assignment $t$ if for $i=1,\ldots, r-1$ the last vertex of $f_i$ is the first vertex of $f_{i+1}$. 
It is easy to check that all monochromatic edges in the coloring constructed by the algorithm have color $r$ and 
every such edge is the last edge of some conflicting $r$-chain. Therefore, if there are no conflicting 
$r$-chains (under assignment $t$), then the coloring produced by the algorithm is proper. 

The \emph{length} of an edge $f\in E$ (under the assignment $t$) is the minimum length of an interval containing all 
the birth times of the vertices of $f$. Let $p=\frac{2 \ln(n)}{n}$. The edge is called \emph{short} if its 
length is smaller than $\frac{1-p}{r}$. The expected number of short edges is less than 
\begin{equation}
\label{eq:shortEd}
	k \;r^{n-2} \; n \left(\frac{1-p}{r}\right)^{n-1} \sim \frac{k}{r \;n}.
\end{equation}

Next we estimate the probability of conflicting $r$-chains in which no edge is short. Let $F=(f_1, \ldots, f_r)$ be an 
$r$-chain, and $x_1, \ldots, x_{r-1}$ be vertices such that $f_i \cap f_{i+1} =\{x_i\}$. Observe that for $F$ to 
be conflicting without short edges, the birth time of each $x_i$ must belong to interval 
$[\frac{i-i p}{r}, \frac{i+(r-i)p}{r}]$ (otherwise the average length of edges to the left or to the right would 
be smaller than $\frac{1-p}{r}$). The probability that vertices $x_1, \ldots, x_{r-1}$ have birth times in 
corresponding intervals is $p^{r-1}$. Once those birth times are fixed, the probability that remaining vertices of 
the chain falls into appropriate intervals is  smaller than (for convenience we put $t(x_0)=0, t(x_r)=1$)
\[
 \prod_{i=0}^{r-1} (t(x_{i+1}) -t(x_i))^{n-2}.
\]
Since the sum of differences in the product is 1, the product is maximized when 
$t(x_{i+1}) -t(x_i)=1/r$ for all $i=0, \ldots, r-1$. Hence the probability that an $r$-chain is 
conflicting is less than $p^{r-1} r^{-r(n-2)}$. As a consequence the expected number of conflicting 
$r$-chains without short edges is less than
\begin{equation}
\label{eq:rchain}
	\frac{2}{r!} (k\; r^{n-2})^r p^{r-1} r^{-r(n-2)} = \frac{2}{r!}  k^r \left(\frac{\ln(n)}{n}\right)^{r-1}.
\end{equation}
For $k<(\frac{n}{\ln(n)})^{\frac{r-1}{r}}$, that number is smaller than $\frac{2}{r!}$. 
Moreover for such $k$, if $n$ is large enough, then the expected number of short edges (\ref{eq:shortEd}) 
is close to zero. In those cases  with positive probability the algorithm produces proper 
$r$-coloring and the theorem follows.
\end{proof}

\begin{corollary}
\label{cor:btaC}
	If there exists a birth time assignment which makes no edge short and creates no conflicting $r$-chains, 
	then random greedy $r$-coloring algorithm produces proper coloring with positive probability 
	(at least the probability of sampling such birth time assignment).
\end{corollary}

\newcommand{\sh}{\mathcal{S}}
\newcommand{\cc}{\mathcal{C}}
\section{Local version}
\begin{proof}[Proof of Theorem \ref{thm:loc}]
Let $H=(V,E)$ be an $n$-uniform hypergraph with maximum edge degree $D-1$.
% JK: I changed to D-1 to keep a bound on dependend chains rD^r, if D was a max deg we would have rD^r + r D^{r-1} 
%because the current edge can also belong to a chain
 To derive sufficient condition for $H$ to be $r$-colorable we apply Lov\'{a}sz Local Lemma to prove that there 
 exits a birth time assignment avoiding short edges and conflicting $r$-chains. Then by Corollary \ref{cor:btaC} 
 random greedy $r$-coloring algorithm produces proper coloring with positive probability. For a  birth assignment 
 function $t$ chosen uniformly at random (as in algorithm \ref{alg}) let us consider following events
\begin{enumerate}
	\item let $\sh_f$  be the event that edge $f$ is short,
	\item let $\cc_s$ be the event that an $r$-chain $s$ is conflicting.
\end{enumerate}
The particular values of $P_1=\Pr(\sh_f)$ and $P_2=\Pr(\cc_f)$ were analysed in Section \ref{sec:rcol}. 
Clearly every event $\sh_f$ is independent of all events $\sh_e$ and $\cc_s$ for $e$ and $s$ disjoint from $f$ 
(analogously for events $\cc_s$). 
Every edge intersects at most $D$ other edges and $r D^r$   different $r$-chains. 
% JK: The bound on depepnded chains should be r D^r, the chain can intersect the current edge in any of r edges of the chain. Analogosly for chains.
Similarly  $r$-chain intersects at most $rD$ edges and $r^2 D^r$ other $r$-chains. Therefore it is sufficient to 
exhibit $x,y \in [0,1)$ for which:
\[
	P_1 \leq x (1-x)^D (1-y)^{r D^r}\;\;\;\text{   and   }\;\;\; P_2 \leq y (1-x)^{rD} (1-y)^{r^2 D^r}
\]
to conclude from Lov\'{a}sz Local Lemma that there exists a birth assignment function which avoids short edges and 
conflicting $r$-chains. 
Choosing $x=1-e^{-a/D}$ and $y=1-e^{-b/(rD^r)}$ the right hand sides of the inequalities become 
$x e^{-(a+b)}$ and $y e^{-r(a+b)}$. A tedious and standard calculations, that we omit here, show that it is possible 
to choose positive $a,b,c$ so that the inequalities are satisfied for all large enough $n$ and 
$D< c (\frac{n}{\ln(n)})^{(r-1)/r} r^n$.

%For small $a,b$ and large $D$, parameter $x$ is close to $a/D$ and $y$ to $b/(rD^r)$ while $P_1\sim n^{-2}/r^{n-1}$ and $P_2 \sim (\log(n)/n)^{r-1}  / r^{r(n-2)}$. 

\end{proof}

\section{Remarks}
\begin{enumerate}
	\item Inequality (\ref{ineq:RS}) is exactly the inequality optimized in \cite{RS00}.
	\item The optimal value for $p$ in the case of 2-coloring have the following combinatorial interpretation.  
	Suppose that the birth time of the last vertex of an edge $e$ is $\frac{1-p}{2}$. Then conditional expected number 
	of conflicting pairs $(e,f)$ is at most $k 2^{n-1} n^{-1} (\frac{1+p}{2})^{n-1}$, which  tends to 1 with $n$, for 
	chosen $k$ and $p$.
	\item The birth times of vertices are used in algorithms only to generate an ordering of $V$, so the same result 
	applies to algorithms which instead choose uniformly at random a permutation of $V$. 
	\item Careful analysis of the algorithm which chooses random permutation can give essentially better bound when 
	the number of vertices is sufficiently small. % (e.g. $o(n^2)$ for 2 coloring). 
In particular, considered algorithms are never worse than choosing equitable partition of vertices into color classes. 
As observed in \cite{RadhStream} for $|V| = O(n^2/\log(n))$ the last strategy with positive probability construct proper 
two coloring of hypergraphs with $\theta(n 2^n)$ edges.
	\item The presented analysis of random greedy algorithm shows that within some intervals the ordering of vertices 
	is irrelevant (e.g. in intervals $B,R$ for 2-coloring). It suggests an equivalent variant of the algorithm which 
	first chooses vertices which fall into these intervals, color these vertices accordingly, and then use random greedy 
	coloring for the remaining ones. Those two phases can be considered as precoloring and random alteration. 
	For 2-coloring  it closely resembles the algorithm of Radhakrishnan and Srinivasan from \cite{RS00} (especially 
	the simplification by Boppana mentioned in the paper).
	
%	Such version of the algorithm was proposed for the general case of $r$ colors by the first author of the paper. 
%	The main version of the algorithm used in this paper was independently proposed by the second author for the case 
%	of two colors. Then it was generalized by the second author to an arbitrary number of colors due to suggestions of 
%	N. Alon, D.A. Shabanov, and A. Rucin\'ski. We decided to choose this version as the main one, as the calculations 
%	here are even a bit simpler than for the other version. 

	\item Random greedy coloring algorithm easily translates to a streaming framework analysed in \cite{RadhStream}.
\end{enumerate}

\vspace{0.5cm}

\textbf{Acknowledgments.}
The present paper is a result of combining two independent papers by the authors each one containing Theorem \ref{thm:rCol} as the main result.
The authors are grateful to N. Alon, A.M. Raigorodskii, A. Ruci\'nski, D.A. Shabanov,  and J. Spencer for discussions and valuable suggestions. 

\bibliographystyle{siam}
\bibliography{pB}

%\begin{thebibliography}{99}
%
%\bibitem{Rai1} A.M. Raigorodskii, D.A. Shabanov, ``The Erd\H{o}s--Hajnal problem on hypergraph colorings, its generalizations, and 
%related problems'', \textit{Russian Math. Surveys}, \textbf{66}:5 (2011), 933--1002.
%
%\bibitem{Sh1} D.A. Shabanov, ``On the chromatic number of finite systems of subsets'', \textit{Math. Notes}, \textbf{85}:6 (2009), 902--905.
%
%\bibitem{Sh2} D.A. Shabanov, ``Improvement of the lower bound in the Erd\H{o}s--Hajnal combinatorial problem'', 
%\textit{Dokl. Math.}, \textbf{79}:3 (2009), 349--350.
%
%
%\end{thebibliography}
\end{document}